\documentclass[aop, noinfoline,reqno]{imsart}
\makeatletter
\def\journal@name{}

\RequirePackage[OT1]{fontenc}
\RequirePackage{amsthm,amsmath}
\RequirePackage[numbers]{natbib}
\RequirePackage[pdftitle={Chen--Stein Method for the Uncovered Set of Random Walk},pdfauthor={Sam Olesker-Taylor and Perla Sousi},bookmarks=true,bookmarksopen=true, bookmarksopenlevel=3,unicode]{hyperref}
\RequirePackage{amssymb}
\RequirePackage{comment}
\RequirePackage{graphicx}
\RequirePackage[font=footnotesize]{caption}
\RequirePackage{subfig}
\RequirePackage{enumerate}
\RequirePackage{comment}
\RequirePackage{colonequals}
\RequirePackage{textgreek}
\RequirePackage[all]{hypcap}
\RequirePackage[normalem]{ulem}
\RequirePackage{xfrac}
\usepackage{color}
\usepackage{tikz}

\usepackage[all]{xy}
\startlocaldefs

\numberwithin{equation}{section}
\newtheorem{theorem}{Theorem}[section]
\newtheorem*{theorem*}{Theorem}
\newtheorem{lemma}[theorem]{Lemma}

\newtheorem{corollary}[theorem]{Corollary}
\theoremstyle{definition}
\newtheorem{definition}[theorem]{Definition}
\newtheorem{remark}[theorem]{Remark}

\def\Z{\mathbb{Z}}

\def\F{\mathcal{F}}

\def\B{\mathcal{B}}
\def\LL{\mathcal{L}}

\def\U{\mathcal{U}}

\renewcommand{\phi}{\varphi}
\renewcommand{\epsilon}{\varepsilon}

\allowdisplaybreaks

\renewcommand{\1}{{\text{\Large $\mathfrak 1$}}}

\newcommand{\var}{\operatorname{var}}
\newcommand{\cov}{\operatorname{Cov}}

\newcommand{\til}{\widetilde}

\newcommand{\tmix}{t_{\mathrm{mix}}}

\newcommand{\tcov}{t_{\mathrm{cov}}}

\newcommand{\ball}[1]{\mathcal{B}(0,#1)}

\newcommand{\pr}[1]{\mathbb{P}\!\left(#1\right)}
\newcommand{\E}[1]{\mathbb{E}\!\left[#1\right]}
\newcommand{\estart}[2]{\mathbb{E}_{#2}\!\left[#1\right]}

\newcommand{\prcond}[3]{\mathbb{P}_{#3}\!\left(#1\;\middle\vert\;#2\right)}
\newcommand{\econd}[2]{\mathbb{E}\!\left[#1\;\middle\vert\;#2\right]}

\newcommand{\norm}[1]{\left\| #1 \right\|}

\newcommand{\tn}{|\kern-.1em|\kern-0.1em|}

\newcommand{\Fexc}{\F_{\rm{exc}}}
\newcommand{\Fout}{\F_{\rm{out}}}

\newcommand{\crw}{C_d}

\endlocaldefs

\begin{document}

\begin{frontmatter}
\title{Chen--Stein Method for the Uncovered Set\\of Random Walk on $\Z_{\lowercase{n}}^{\lowercase{d}}$ for $\lowercase{d} \ge 3$}
\runtitle{Uncovered set of random walk}
\begin{aug}

\author{Sam Olesker-Taylor}
and
\author{Perla Sousi}

\runauthor{Sam Olesker-Taylor and Perla Sousi}

\address{
	Sam Olesker-Taylor\\
	University of Cambridge\\
	Cambridge, UK\\
	sam.ot@posteo.co.uk%
}

\address{
Perla Sousi\\
University of Cambridge\\
Cambridge, UK\\
p.sousi@statslab.cam.ac.uk%
}

\affiliation{University of Cambridge}
\end{aug}

\begin{abstract}
Let $X$ be a simple random walk on $\Z_n^d$ with $d\geq 3$ and let $\tcov$ be the expected cover time. We consider the set  $\U_\alpha$ of points of $\Z_n^d$ that have not been visited by the walk by time $\alpha \tcov$ for $\alpha\in (0,1)$. It was shown in~\cite{JasonPerla} that there exists $\alpha_1(d)\in (0,1)$ such that for all $\alpha>\alpha_1(d)$ the total variation distance between the law of the set $\U_\alpha$ and an i.i.d.\ sequence of Bernoulli random variables indexed by $\Z_n^d$ with success probability $n^{-\alpha d}$ tends to $0$ as $n \to \infty$. In~\cite{JasonPerla} the constant $\alpha_1(d)$ converges to $1$ as $d\to\infty$. In this short note using the Chen--Stein method and a concentration result for Markov chains of Lezaud~\cite{Lezaud} we greatly simplify the proof of~\cite{JasonPerla} and find a constant $\alpha_1(d)$ which converges to $3/4$ as $d\to\infty$. We prove analogous results for the high points of the Gaussian free field. 
\end{abstract}

\thispagestyle{empty}

\setattribute{keyword}{AMS}{AMS 2010 subject classifications:}


\setattribute{keyword}{note}{Note:}
\begin{keyword}[class=note]
	This article was originally published under the first named author's previous name, ``Sam Thomas''
\end{keyword}

\date{\today}

\end{frontmatter}

\maketitle

\section{Introduction}
\label{sec:intro}

Let $X$ be a simple random walk on the torus $\Z_n^d$ with $d \geq 3$ started from its stationary distribution.  
For each~$x \in \Z_n^d$ we let 
\[\tau_x = \min\{t \geq 0 : X(t) = x\}\]
be the first time that $X$ visits $x$.  For~$t\geq 0$ we define the process
$(U_x(t))_x$ and the \emph{uncovered} set $\U(t)$ respectively by
\[U_x(t) = \1(\tau_x>t) \quad \text{for}\quad x\in 
\Z_n^d \quad \text{and} \quad
\U(t) = \{ x \in \Z_n^d : U_x(t) = 1\}.\]
The expected cover time $\tcov$ is given by 
\[\textstyle
\tcov = \max_x\estart{\max_y\tau_y}{x},
\]
where we write $\mathbb{E}_x$ and $\mathbb{P}_x$ to indicate that the random walk  starts from $x$. We write $\mathbb{E}$ and~$\mathbb{P}$ when $X$ starts from stationarity. 

We recall that the total variation distance between two measures $\mu$ and $\nu$ is given by 
\[
\norm{\mu-\nu}_{\rm{TV}}= \max_{A\subseteq \Z_n^d} |\mu(A)-\nu(A)|.
\]
For any $\alpha>0$ let $p_{\alpha,n}\in (0,1)$ be a parameter to be defined precisely later (in the proof of Theorem~\ref{thm:mainresult}) which satisfies 
\[
p_{\alpha,n}=n^{-\alpha d}(1+o(1)).
\]
Let $t_*$ be a time to be defined precisely later (see~\eqref{eq:defoftstar}) which satisfies
\[
t_* = \tcov \cdot (1+o(1)).
\]
Finally let $\nu_{\alpha,n}$ be the law of $\{x\in \Z_n^d: Z_x=1\}$ where $(Z_x)_x$ is an i.i.d.\ sequence of Bernoulli random variables with parameter $p_{\alpha, n}$. 
The following theorem was shown in~\cite{JasonPerla}.

\begin{theorem}[{\cite[Theorem~1.1]{JasonPerla}}]\label{thm:perlajason}
	For all $d\geq 3$, there exist $0<\alpha_0(d)<\alpha_1(d)<1$ so that for all $\alpha<\alpha_0(d)$ 
\[
	\norm{\LL(\U(\alpha t_*)) - \nu_{\alpha,n} }_{\rm{TV}} = 1-o(1) \ \text{ as } n\to\infty,
	\]
while for all $\alpha>\alpha_1(d)$
\begin{align*}\label{eq:alpha1}
	\norm{\LL(\U(\alpha t_*)) - \nu_{\alpha,n} }_{\rm{TV}} =o(1) \ \text{ as } n\to\infty.
	\end{align*}
\end{theorem} 
	The existence of $\alpha_1(d)$ was the main challenge in~\cite{JasonPerla}, while the existence of $\alpha_0(d)$ followed by counting the number of neighbouring points in the uncovered set. In~\cite[Question~1]{JasonPerla} the authors ask whether there is a phase transition, i.e.\ whether $\alpha_0(d)$ and~$\alpha_1(d)$ can be chosen to be equal. They obtained $\alpha_0(d) = (1+p_d)/2$, where $p_d$ is the return probability to $0$ for simple random walk on $\Z^d$, while their constant $\alpha_1(d)\to 1$ as $d\to\infty$. 
	
	Our contribution in the present paper is to give a much simpler proof of the existence of the constant $\alpha_1(d)$ and moreover to show that $\alpha_1(d)$ can be chosen to be bounded away from $1$ as $d\to\infty$ as the following theorem shows.

\begin{theorem}\label{thm:mainresult}

	For all $d\geq 3$, if $\alpha> \frac34 (d-\frac23)/(d-1)$, then
	\[
	\norm{\LL(\U(\alpha t_*)) - \nu_{\alpha,n} }_{\rm{TV}} = o(1) \ \text{ as } n\to\infty.
	\]
\end{theorem}

In~\cite{JasonPerla} as a corollary of Theorem~\ref{thm:perlajason} it was shown that the same uniformity statement holds when we wait for the first time the uncovered set contains $n^{d-\alpha d}$ points. Using our improved bound on $\alpha_1(d)$ one can use exactly the same proof as in~\cite{JasonPerla} to obtain the same result for this larger range of $\alpha$.

We prove Theorem~\ref{thm:mainresult} in Section~\ref{sec:totalvar} using the Chen--Stein method. Using the same method we also prove an analogous result for the high points of the Gaussian free field that we explain now.  

Let $(\phi_x)$ be the discrete Gaussian free field on $[0,n]^d\cap \Z^d$ with $0$ boundary condition on~$\partial [0,n]^d$. That is, $\phi$ is a zero mean Gaussian process indexed by the vertices of $[0,n]^d\cap \Z^d$ with covariance 
\[\textstyle
\cov(\phi_x,\phi_y) = \estart{\sum_{i=0}^{\tau}\1(X_i=y)}{x},
\]
where $X$ is a simple random walk on $[0,n]^d\cap \Z^d$ and $\tau$ is the first time it hits $\partial [0,n]^d\cap \Z^d$.
Fix $\delta\in (0,1)$. For each $\alpha\in (0,1)$ consider the set of $\alpha$-high points 
\[
\mathcal H_\alpha = \left\{x\in [\delta n, (1-\delta)n]^d\cap \Z^d: \phi_x \geq \sqrt{2\alpha d G(0)\log n}\right\},
\]
where $G$ stands for the Green's function for simple random walk on $\Z^d$ and we write $G(0)$ for~$G(0,0)$. 
Finally, let $\til{\nu}_{\alpha,n}$ be the law of $\{x\in [\delta n, (1-\delta)n]^d\cap \Z^d: \til{Z}_x=1\}$, where $(\til{Z}_x)_x$ is an independent sequence of Bernoulli random variables satisfying for all $x\in  [\delta n, (1-\delta)n]^d\cap \Z^d$ 
\[
\pr{\til{Z}_x=1} = \pr{\phi_x \geq \sqrt{2\alpha d G(0)\log n}}.
 \]

\begin{theorem}\label{thm:gff}
For all $d\geq 3$, if $\alpha>\tfrac{3}{4}(d-\tfrac{2}{3})/(d-1)$, then 
\[
\norm{\LL(\mathcal{H}_\alpha) - \til{\nu}_{\alpha,n} }_{\rm{TV}} = o(1) \ \text{ as } n\to\infty.
\]	
\end{theorem}
\begin{remark}
	\rm{
	We remark that for all $\alpha<\alpha_0(d)$, where $\alpha_0(d)=(1+p_d)/2$ as above, 
	\[
	\norm{\LL(\mathcal{H}_\alpha) - \til{\nu}_{\alpha,n}}_{\rm{TV}} = 1-o(1) \ \text{ as } n\to\infty.
	\]
	Indeed, this follows by using exactly the same approach as the proof of the lower bound of the proof of Theorem~1.1 of~\cite{JasonPerla}. In particular, one shows that the number of neighbouring points contained in a set is a distinguishing statistic for the two random subsets of $\Z_n^d$. 
	}
\end{remark}

We prove Theorem~\ref{thm:gff} in Section~\ref{sec:gff} where we also recall the statement of the Chen--Stein method. The proof of this theorem serves as a warm-up for the proof of Theorem~\ref{thm:mainresult}.

A statement similar to Theorem~\ref{thm:gff} in the case $\alpha=1$ has been established in a recent work by~\cite{ChiariniCiprianiSubhra} using also the Chen--Stein method. In many instances, an approach in the study of GFF can be used to study random walks, or vice versa.

\emph{Notation.}
For functions $f$ and $g$ we write $f(n) \lesssim g(n)$ if there exists a constant $c > 0$ such that $f(n) \leq c g(n)$ for all $n$.  We write $f(n) \gtrsim g(n)$ if $g(n) \lesssim f(n)$.  Finally, we write $f(n) \asymp g(n)$ if both $f(n) \lesssim g(n)$ and $f(n) \gtrsim g(n)$.  

\section{High points of the GFF}\label{sec:gff}

We start this section by recalling the statement of the Chen--Stein method~\cite[Theorem~3]{Arratia}, see also~\cite{Chen, Stein}.

\begin{theorem}[Chen--Stein]\label{thm:chenstein}
Let $I$ be a finite set of indices and let $(X_t)_{t\in I}$ and $(X'_t)_{t\in I}$ be two processes such that $X_t$ has the same distribution as $X'_t$ for all $t\in I$ and $X'$ is an independent Bernoulli process. For every $t\in I$ we choose a set $\B_t\subseteq I$ with $t\in \B_t$ and call it the neighbourhood of $t$. Writing $p_t = \pr{X_t=1}$ and $p_{rs}=\pr{X_r=1, X_s =1}$ we define
\begin{align*}
	b_1 = \sum_{t} \sum_{s \in \B_t} p_t p_s,
\quad
	b_2 = \sum_{t} \sum_{s\in \B_t\setminus\{t\}} p_{st}
\quad \text{and}\quad 
	b_3 = \sum_{t} \E{\left|\econd{X_t - p_t}{X_s, s \notin \B_t} \right|}.
\end{align*}
Then we have 
\[
\norm{\LL((X_t)_{t\in I}) - \LL((X'_t)_{t\in I})}_{\rm{TV}} \leq 8(b_1 + b_2 + b_3).
\]	
\end{theorem}

\begin{proof}[\bf Proof of Theorem~\ref{thm:gff}]
Let $\alpha > (1+p_d)/2$, where $p_d$ is the return probability to $0$ for simple random walk on $\Z^d$ and let $\gamma>0$ to be determined later. In order to apply Theorem~\ref{thm:chenstein} for every $x\in [\delta n, (1-\delta)n]^d\cap \Z^d$ we define its neighbourhood to be $\B_x= \B(x, n^\gamma)$. We now need to bound the quantities $b_1, b_2$ and $b_3$. To simplify notation we write $t=\sqrt{2\alpha d G(0)\log n}$. 
Using that $\phi_x$ has the Gaussian distribution with mean $0$ and variance $G(0)+O(n^{-(d-2)})$ we get 
	\[
	p_x=\pr{\phi_x\geq t} \asymp \frac{n^{-\alpha d}}{\sqrt{\log n}}.
	\]
So it immediately follows that $b_1\lesssim n^d n^{\gamma d}n^{-2\alpha d} $. Regarding the term $b_2$, we separate into two cases depending on whether or not $\norm{x-y}\leq n^\zeta$ for some $\zeta$ to be determined. Then using that $G(0,x)\asymp \norm{x}^{2-d}$ as $x\to \infty$ (see~\cite[Theorem~4.3.1]{LawLim}) and $G(0,x)\leq p_d G(0)$ for all $x\neq 0$ we get 
\[
b_2\lesssim n^{d} n^{\zeta d} n^{-2\alpha d/(1+p_d)} + n^d n^{\gamma d} n^{-2\alpha d}.
\]
Since $\alpha>\tfrac{1+p_d}{2}$, we can choose $\zeta$ sufficiently small and take $\gamma<2\alpha -1$ to get~$b_2 = o(1)$. 

It remains to bound the term $b_3$. We write $\Fout$ for the $\sigma$-algebra generated by $(\phi_y)_{y\notin \B_x}$.
By the Markov property for the Gaussian free field, we have  
\[
\prcond{\phi_x \geq t}{\Fout}{} = \prcond{h_x+ \xi_x \geq t}{\xi_x}{},
\]
where $\xi_x = \sum_{y\in \partial \B_x}p_{xy}\phi_y$ with $p_{xy}$ being the probability that a simple random walk started from $x$ first hits $\partial \B_x$ (outer boundary of $\B_x$) at $y$ and $h$ is an independent GFF in $\B_x$ with $0$ boundary condition on $\partial \B_x$. 
Write
\[
	v_h^2 = \var(h_x),
\quad
	v_\phi^2 = \var(\phi_x)
\quad\text{and}\quad
	v_\xi^2 = \var(\xi_x).
\]
Note that $v_h^2 = G(0) + O(n^{-\gamma(d-2)})$ and $v_\phi^2= G(0) + O(n^{-(d-2)})$, while $v_\xi^2 = O( n^{-\gamma(d-2)} )$.
We now have 
\begin{align}\label{eq:boundonabsolutevalue}
\nonumber	|&\prcond{h_x+ \xi_x \geq t}{\xi_x}{} - \pr{\phi_x\geq t}|
	\\ &\leq  |\prcond{h_x\geq t-\xi_x}{\xi_x}{} - \pr{h_x\geq t}| + |\pr{h_x\geq t} - \pr{\phi_x\geq t}|.
\end{align}
For the first term in the sum above taking expectations we get 
\begin{align}\label{eq:manyeq}
\begin{split}
&\E{	|\prcond{h_x\geq t-\xi_x}{\xi_x}{} - \pr{h_x\geq t}|} \leq \pr{|\xi_x|\geq t} + \E{|\xi_x| \exp\left( -(t-|\xi_x|)^2/(2v_h^2) \right)} \\
&\lesssim e^{-t^2/(2v_\xi^2)} +v_\xi \exp\left(-t^2/(2v_h^2) \right)\left(\E{\exp\left( 2t|\xi_x|/v_h^2  \right)}\right)^{1/2} \lesssim n^{-\alpha d} n^{-\gamma (d-2)/2},
\end{split}
\end{align}
where for the second inequality we used the Cauchy--Schwarz inequality. We now turn to the second term in~\eqref{eq:boundonabsolutevalue}. As noted above, we have $\Delta^2=|v_h^2-v_\phi^2|=O(n^{-\gamma (d-2)})$. Suppose without loss of generality that $v_h\geq v_\phi$. Let $\mathcal{N}$ be a $N(0,1)$ random variable independent of $\phi_x$. Then we have 
\begin{align*}
	|\pr{h_x\geq t} - \pr{\phi_x\geq t}| &= \left| \pr{\phi_x + \Delta \mathcal{N}\geq t} - \pr{\phi_x\geq t}   \right| \\
	&\leq \E{\left| \prcond{\phi_x + \Delta \mathcal{N}\geq t}{\mathcal{N}}{}   \right| - \pr{\phi_x\geq t}} \lesssim n^{-\alpha d} n^{-\gamma (d-2)/2},
\end{align*}
where for the last inequality we used exactly the same reasoning as for~\eqref{eq:manyeq}. This together with~\eqref{eq:manyeq} and~\eqref{eq:boundonabsolutevalue} gives  
\[\textstyle
b_3= \sum_{x} \E{\left|\prcond{\phi_x \geq t}{\Fout}{} - \pr{\phi_x\geq t}  \right|}\lesssim n^d n^{-\alpha d} n^{-\gamma (d-2)/2}.
\]
Therefore, if $\alpha > \tfrac34 ( d - \tfrac23 ) / ( d - 1)$ and $\gamma=2\alpha -1-\epsilon$, for $\epsilon$ sufficiently small, then $b_3=o(1)$ and this concludes the proof.	
\end{proof}

\section{Outline and preliminaries}\label{sec:excursions} 

We start by giving a brief outline of the proof of Theorem~\ref{thm:mainresult}. The strategy of the proof is to define another set, called $\overline{\U}$ below, that can be coupled with $\U(\alpha t_*)$ so they agree with high probability and for which we can apply more easily the Chen--Stein method to show that it is close to the distribution $\nu_{\alpha,n}$. 

For every point $x$ we look at the number of excursions across an annulus of suitably defined radii that the random walk completes by $\alpha t_*$. In~\cite{JasonPerla} it was shown that these numbers of excursions are highly concentrated. We include $x$ in the set $\overline{\U}$ if $x$ remains uncovered after the walk has completed the appropriate number of excursions in the annulus centred at $x$. This way, the events that the points separated by the annuli are uncovered become uncorrelated, hence making it easier to apply the Chen--Stein method. Using the concentration of the number of excursions, we show in Lemma~\ref{lem:comparison} that $\overline{\U}=\U(\alpha t_*)$ with high probability. 

In the remainder of this section we give some definitions and preliminary results that will be used in the proof of Theorem~\ref{thm:mainresult} in Section~\ref{sec:totalvar}.

We write~$\B(x,r)$ for the closed Euclidean ball in $\Z_n^d$ centred at $x$ of radius $r$, i.e.
\[\textstyle
\B(x,r) =\left\{y\in\Z_n^d: \, \sum_i(|y_i-x_i
|\wedge (n-|y_i-x_i|)\bigr)^2\leq r^2\right\},
\]
For a set~$A$ we define the boundary $\partial A$ to be the outer boundary, i.e.
\[
\partial A = \{y\notin A: \, \exists \,x \in A \, \text{adjacent to } y\}.
\]
\begin{definition}\label{def:excursions}
\rm{
For $r<R$ and $x\in \Z_n^d$ we define the following sequence of stopping times
\begin{align*}
\rho_0^x=\inf\{t\geq 0: X(t) \in \partial \B(x,r)\}, \\
\til{\rho}_0^x = \inf\{t\geq \rho_0^x: X(t)\notin \B(x,R)\}
\end{align*}
and inductively we set 
\begin{align*}
\rho_{k+1}^x &= \inf\{t\geq \til{\rho}_{k}^x: X(t) \in \partial \B(x,r)\} \\
\til{\rho}_{k+1}^x &= \inf\{ t\geq \rho_{k+1}^x: X(t) \notin \B(x,R)\}.
\end{align*}
When $x=0$, we omit the superscript. 
We call a path of the random walk trajectory an {\it{excursion}} if it starts from $\partial \B(x,R)$ and it comes back to $\partial \B(x,R)$ after hitting $\B(x,r)$.

We define $N_x(r,R,t)$ to be the total number of excursions across the annulus $\B(x,R)\setminus \B(x,r)$ before time~$t$ after the first time that $X$ hits $\partial \B(x,R)$, i.e.
\[\textstyle
N_x(r,R,t) = \min\left\{k\geq 0: \sum_{i=1}^{k}(\til{\rho}_i^x-\til{\rho}_{i-1}^x)\geq  t\right\}.
\]
}
\end{definition}

We next recall \cite[Lemma~2.2]{JasonPerla} proved in the Appendix of~\cite{JasonPerla} showing that the mixing time of the exit points of the excursions mix in time of order $1$.

\begin{lemma}[{\cite[Lemma~2.2]{JasonPerla}}]
\label{lem:coupling}
Let $R \geq 10r$ and let $Y_j$ be the exit point of the $j$-th excursion across $\ball{R}\setminus\ball{r}$.  Then $(Y_j)_j$ is a finite state space Markov chain. Let $\til{\pi}$ be its stationary distribution. Then the mixing time of the chain is of order $1$, i.e.\ there exists $k_0<\infty$ such that $\tmix = k_0$ and $k_0$ only depends on $d$.  
\end{lemma}

\begin{corollary}\label{cor:jointprocess}
	The process $(Y_{i-1}, Y_i)$ mixes in time of order $1$ and its stationary distribution is given by $\nu(x,y) = \til{\pi}(x) P(x,y)$, where $P$ is the transition matrix of $Y$. Moreover, there exist positive constants $c_1$ and $c_2$ so that for all $(x,y)\in \partial \B(0,R)\times \partial \B(0,R)$ the measure $\nu$ satisfies 
	\[
	c_1R^{-2(d-1)}\leq \nu(x,y)\leq c_2R^{-2(d-1)}.
	\]
\end{corollary}

\begin{proof}[\bf Proof]
	By the definition of total variation distance it is easy to show that for all times $t$ we have 
	\[
	\|\LL(Y_t) - \til{\pi} \|_{\rm{TV}} = \|\LL(Y_t,Y_{t+1}) - \nu \|_{\rm{TV}}.
	\]
	This together with Lemma~\ref{lem:coupling} shows that the mixing time of $(Y_{i-1},Y_i)$ is of order $1$. 
	
	For the second claim we use that by Harnack's inequality there exists a universal constant~$c$ so that for all $(x,y) \in \partial \B(0,R)\times \partial \B(0,R)$ 
	\[
 \frac{1}{cR^{d-1}}\leq 	P(x,y) \leq \frac{c}{R^{d-1}}.
	\]
	This now implies that $\til{\pi}(x)\asymp R^{-(d-1)}$, and hence this completes the proof. \end{proof}

\begin{definition}\label{def:trR}\rm{
For $R\geq 10r$ we let 
\[
T_{r,R} = \estart{\til{\rho}_1 - \til{\rho}_0}{\til{\pi}},
\]
i.e.\ $T_{r,R}$ is the expected length of the excursion
when the walk is started on $\partial\ball{R}$ according to the stationary distribution $\til{\pi}$ of the exit points of the excursions across the annulus $\B(0,R)\setminus\B(0,r)$ as given in Lemma~\ref{lem:coupling}. 
}
\end{definition}
The following lemma was proved in~\cite{JasonPerla}. The main idea behind the proof is to allow enough time between excursions so that the walk mixes and this essentially gives an almost i.i.d.\ sequence of excursion lengths.

\begin{lemma}[{\cite[Lemma~2.4]{JasonPerla}}]
\label{lem:n1}
For each $\psi\in (0,1/2)$ there exists $n_0\geq 1$ and a positive constant $c$ such that for all $n\geq n_0$ the following is true. Suppose that $n/4\geq R\geq 10r$ and $t\asymp n^d\log n$. Then for all~$\delta>0$ such that~$\delta r^{d-2} n^{-\psi-1/2} \geq 1$ and~$\delta n^{\psi} \geq 1$, for all $x$ we have
\[
\pr{N_x(r,R,t) \notin [A,A']} \lesssim  n^\psi e^{-c\delta^2 r^{d-2}/n^\psi} + e^{-cn^\psi},
\]
where $A=t/((1+\delta)T_{r,R})$ and $A'=t/((1-\delta) T_{r,R})$.
\end{lemma}

We finally recall another standard result that was proved in the Appendix of~\cite{JasonPerla} which shows that conditioning on the entrance and exit points of an excursion does not affect the probability of hitting the centre. The proof is an easy consequence of Harnack's inequality.

\begin{lemma}[{\cite[Lemma~3.2]{JasonPerla}}]\label{lem:hittingprob}
There exists a constant $C_d > 0$ depending only on $d$ such that the following is true.  Let $n/4 \geq R\geq 2r$ such that both $r$ and $R$ tend to infinity as $n\to \infty$. We denote by~$\tau_R$ the first hitting time of $\partial\B(0,R)$ and by $\tau_0$ the first hitting time of $0$. Then for all $x\in \partial\B(0,r)$ and all $y\in \partial\B(0,R)$ we have
\begin{align*}
\prcond{\tau_0<\tau_R}{X(\tau_R)=y}{x} = \frac{\crw}{r^{d-2}}\left(1+O\left(\frac{r}{R} \right) +O\left(\frac{1}{r^2} \right)  \right).
\end{align*}
The constant $\crw$ is given by $c_d/G(0)$, where $c_d$ is the constant from~\cite[Theorem~4.3.1]{LawLim} and $G$ is the Green's function for simple random walk on $\Z^d$.
\end{lemma}

\begin{remark}\rm{
To avoid confusion, we emphasize that $\tau_x, 
\tau_y$ and $\tau_z$ will always refer to hitting times of a point, while $\tau_r$ and $\tau_R$ to hitting times of boundaries of balls.
}
\end{remark}

We recall $p_d$ is the probability that a simple random walk on~$\Z^d$ started from $0$ returns to~$0$. 
For $d=3$, it is well-known (see e.g.~\cite{Spitzer}) that $p_3\approx 0.34$. It is also easy to see that $p_d\to 0$ as $d\to \infty$. Note that $p_d$ is equal to the probability that a simple random walk in~$\Z^d$ starting from~$0$ visits a given neighbour of~$0$ before escaping to~$\infty$.

\begin{lemma}[{\cite[Lemma~3.7]{JasonPerla}}]\label{lem:2pointestimate}
Let $n/4 \geq R>2r \to \infty$ and $x,y\in \Z_n^d$ satisfying $\norm{x-y}=o(r)$. 
We denote by $\tau_R$ the first hitting time of~$\B(x,R)$ and by $\tau_x$ (resp.\ $\tau_y$) the first hitting time of~$x$ (resp.~$y$).
Then for all $a\in \partial\B(x,r)$ and all $b\in \partial\B(x,R)$ we have
\begin{align*}
\prcond{\tau_x\wedge\tau_y<\tau_R}{X(\tau_R)=b}{a} &\geq \frac{2\crw}{(1+p_d)r^{d-2}}\left( 1 +o(1)+O\left(\frac{r}{R} \right) + O\left(\frac{1}{r^2} \right) \right).
\end{align*}
\end{lemma}

\section{Uncovered set}\label{sec:totalvar}

In this section we prove Theorem~\ref{thm:mainresult}. 
For $\alpha>\alpha_0(d)=(1+p_d)/2$, we now fix $\gamma=2\alpha-1 -\epsilon$, with $\epsilon>0$ sufficiently small, and set $r=n^{\gamma(1-\epsilon)}$ and $R=n^\gamma$. 

Recall the definition of the times $(\rho_i)$ and $(\til{\rho}_i)$ from Definition~\ref{def:excursions}. 
Next we define a function $f: \partial \B(0,R)\times \partial\B(0,R)\to [0,1]$ given by 
\begin{align}\label{eq:defoff}
f(x,y) = \prcond{\tau_0>\til{\rho}_0}{X_{\til{\rho}_0}=y}{x},
\end{align}
i.e.\ this is the probability that $0$ is not hit in an excursion of the walk starting from $x$ and conditioned to exit at $y$.

Recall the definition of the chain $Y$ from Lemma~\ref{lem:coupling} as the sequence of exit points of the excursions and $\nu$ stands for its invariant distribution. Let 
\[
\Fexc=\sigma(Y_i:i\geq 1) \  \text{ and } \ m=-\estart{\log f(Y_0,Y_1)}{\nu}.
\]
 We take $\delta=r^{(2-d)/2}n^\psi$ for $\psi>0$ sufficiently small and define 
\begin{align}\label{eq:defoftstar}
t_*= \frac{1}{m}\cdot \log (n^d) T_{r,R} \quad \text{ and }\quad  A = \frac{\alpha}{(1+\delta)} \cdot\frac{t_*}{T_{r,R}}.
\end{align}

\begin{lemma}\label{lem:expreform}
	As $n\to\infty$ we have 
	\[
	m=C_d \cdot n^{-\gamma(1-\epsilon) (d-2)}(1+O(n^{-\gamma \epsilon})) \quad \text{ and } \quad t_* = \tcov (1+o(1)).
	\]
	In particular, 
	\[
	A = \alpha C_d^{-1}\cdot  n^{\gamma(1-\epsilon)(d-2)} \log (n^d)\cdot (1+O(\delta)+O(n^{-\gamma \epsilon})).
	\]
\end{lemma}

\begin{proof}[\bf Proof]
Since $r=R^{1-\epsilon}$, it follows from Lemma~\ref{lem:hittingprob} that for all $x$ and $y$ we have 
\begin{align*}
	f(x,y) =1- \frac{C_d}{r^{d-2}} \cdot (1+O(1/r^2)+O(R^{-\epsilon})).
\end{align*}
Therefore for all $x$ and $y$ we obtain
\[
-\log f(x,y) = \frac{C_d}{r^{d-2}} \cdot (1+O(1/r^2)+O(1/r^{d-2})+O(R^{-\epsilon})),
\]
and hence, substituting the value of $r$ proves the first equality of the lemma. 

The proof of the second equality follows from~\cite[Lemma~4.1]{JasonPerla} which is proved in the Appendix of \cite{JasonPerla}.
The last equality follows from the definition of $A$ from~\eqref{eq:defoftstar}.
\end{proof}

For every $x$ let $\sigma_x$ be the first time that the walk has completed $A$ excursions across the annulus $\B(x,R)\setminus \B(x,r)$, i.e.
\[
\sigma_x=\inf\{t\geq 0: N_x(r,R,t)\geq A\}.
\]
We also define
\[
\til{\tau}_x=\inf\left\{t\geq \tau_{\partial \B(x,R)}: X_t=x\right\}\quad \text{ and } \quad Q_x=\1(\til{\tau}_x>\sigma_x),
\]
and consider the set $\overline{\U}=\{x: Q_x=1\}$.

\begin{lemma}\label{lem:q0precise}
There exists a positive constant $c$ so that for all $\eta \in (0,1)$ we have 
	\begin{align*}
		\pr{\econd{Q_0}{\Fexc} \notin \left( \exp(-m(1+\eta)A), \exp(-m(1-\eta)A) \right)}\lesssim R^{d-1}\exp(-c\eta^2 A).
	\end{align*} 
		In particular, as $n\to\infty$ we have 
		\[
		\E{Q_0} = \frac{1}{n^{\alpha d}} \cdot (1+o(1)). 
		\]
\end{lemma}

\begin{proof}[\bf Proof]

Since after conditioning on $\sigma((Y_i)_{i\geq 1})$ the events that $0$ is hit in the $i$-th excursion become independent, we obtain
\[\textstyle
\econd{Q_0}{\Fexc} = \prod_{i=1}^{A} f(Y_{i-1},Y_i),
\]
where the function $f$ was defined in~\eqref{eq:defoff}. Taking logarithms we get  
\[\textstyle
\log \econd{Q_0}{\Fexc} = \sum_{i=1}^{A} \log f(Y_{i-1},Y_i).
\]
Using now \cite[Theorem~1.1 and Remark~1]{Lezaud} and Corollary~\ref{cor:jointprocess}, for positive constants $c$ and $C$ and all $\eta\in (0,1)$ we get
\begin{align*}
	\pr{\left|\frac{1}{A}\sum_{i=1}^{A} \frac{\log f(Y_{i-1},Y_i)}{-m} - 1\right| \geq \eta} \leq C R^{d-1}\exp\left(-c\eta^2A \right).
\end{align*}
This proves the first statement of the lemma.

We turn to proving the second statement. Let $F$ be the event 
\[
F=\left\{ \econd{Q_0}{\Fexc} \in \left( \exp(-m(1+\eta)A), \exp(-m(1-\eta)A) \right)\right\}.
\]
Taking $\eta \asymp n^{-\gamma(d-2)(1-\epsilon)/2} \log n$ and using the definitions of $m$ and $A$ we thus deduce 
\begin{align*}
	\E{Q_0} &= \E{\econd{Q_0}{\Fexc}\1(F)} + \E{\econd{Q_0}{\Fexc}\1(F^c)} \\&
	\leq \exp\left(-m(1-\eta)A \right) + CR^{d-1}\exp\left(-c\eta^2A \right)
\\& \leq  \exp\left( -\alpha \log (n^d)(1+O(\delta)+O(\eta) )  \right) + \exp\left(-c'(\log n)^3\right)\\&
=n^{-\alpha d}(1+o(1)),
\end{align*}
where $c'$ is a positive constant. 
For the lower bound we get
\begin{align*}
	\E{Q_0} \geq \exp\left(-m(1+\eta)A \right) \cdot \left(1-\exp(-c\eta^2 A)\right) = n^{-\alpha d}(1+o(1)),
\end{align*}
where the equality again follows from Lemma~\ref{lem:expreform}.
\end{proof}

\begin{lemma}\label{lem:comparison}
	For $\alpha>\frac34(d-\frac 23)/(d-1)$ and all $\epsilon$ and $\psi$ (in the definitions of $r$ and $\delta$) sufficiently small as $n\to\infty$ we have 
	\[
	\pr{\U(\alpha t_*) = \overline{\U}} = 1-o(1).
	\]
\end{lemma}

\begin{proof}[\bf Proof]

We have by Lemma~\ref{lem:n1} that 
\begin{align*}
\pr{\U(\alpha t_*) \nsubseteq \overline{\U}}= \pr{\exists \ x: N_x(r,R,\alpha t_*)<A} \leq n^d \pr{N_x(r,R,\alpha t_*)<A} = o(1).
\end{align*}	
Set $A'=\alpha t_*/(T_{r,R}(1-\delta))$.
For each $x$ let $\til{\sigma}_x$ be the first time the walk has completed $A'$ excursions across the annulus $\B(x,R)\setminus \B(x,r)$. Then again by Lemma~\ref{lem:n1} we get
\begin{align*}\textstyle
	\pr{\min_x\til{\sigma}_x>\alpha t_*} = 1-o(1).
\end{align*}
We now obtain
\begin{align*}\textstyle
	\pr{ \overline{\U}\nsubseteq\U(\alpha t_*)}\leq \pr{\min_x\til{\sigma}_x>\alpha t_*, \ \exists \ x: \tau_x>\sigma_x \text{ and } \tau_x<\alpha t_*} + o(1).
\end{align*}
The first term on the right hand side above can be upper bounded by 
\begin{align*}
&\textstyle\sum_x \prcond{\tau_x<\alpha t_*, \til{\sigma}_x>\alpha t_*}{\tau_x> \sigma_x}{} \pr{\tau_x>\sigma_x}\\ &\textstyle\leq 
\sum_x\pr{x\text{ hit during } A'-A \text{ excursions}}\pr{\tau_x>\sigma_x}	.
\end{align*}
Using Lemma~\ref{lem:hittingprob} for a positive constant $c$ we have 
\[
\pr{x\text{ hit during } A'-A \text{ excursions}}\leq  1 - \left(1-\frac{c}{n^{\gamma (1-\epsilon)(d-2)}} \right)^{2\delta A/(1-\delta)}\asymp \delta \log n.
\]
We therefore deduce
\begin{align*}\textstyle
	\sum_x\pr{x\text{ hit during } A'-A \text{ excursions}}\pr{\tau_x>\sigma_x}\lesssim \delta (\log n) \E{|\overline{\U}|}.
\end{align*}
From Lemma~\ref{lem:q0precise} we immediately get $\E{|\overline{\U}|} \asymp n^{d-\alpha d}$, and hence the above bound is equal to $n^{d-\alpha d} \delta \log n$, which by the choice of $\delta$ is equal to
 \[
 n^{d-\alpha d} \cdot n^{\gamma(1-\epsilon)(2-d)/2+\psi} \cdot (\log n).
 \]
Since $\alpha>\frac34(d-\frac23)/(d-1)$, by taking $\epsilon$ and $\psi$ sufficiently small the quantity above becomes~$o(1)$ as $n\to\infty$ and this concludes the proof.
\end{proof}

\begin{lemma}\label{lem:b2quantity}
Let $x,y\in \Z_n^d$ and let $0<\zeta< \gamma(1-\epsilon)$. 

{\rm (a)} If $\norm{x-y}\leq n^\zeta$, then 
\[
\E{Q_xQ_y}\lesssim n^{-2\alpha d/(1+p_d) + o(1)}.
\]
{\rm (b)} If $ \norm{x-y}\geq n^\zeta$, then 
\[
\E{Q_xQ_y}\lesssim n^{-2\alpha d +o(1)}.
\]
\end{lemma}

\begin{proof}[\bf Proof]
(a) Let $E$ be the number of excursions across $\B(x,R)\setminus \B(x,r)$ that the walk has completed by time $\sigma_y$. We set 
\[
S=\frac{1}{(1+\delta)}\cdot \frac{\alpha t_* (1-\delta)^2}{T_{r,R}}.
\]
Then we have
\begin{align}\label{eq:upperboundonqx}
	\E{Q_xQ_y} \leq \E{Q_xQ_y\1(E>S)} + \pr{E\leq S}.
\end{align}
Since $S<A$, we can bound $\E{Q_xQ_y\1(E>S)}$ from above
by the probability that neither~$x$ nor $y$ are hit in $S$ excursions across $\B(x,R)\setminus\B(x,r)$.
	Using Lemmas~\ref{lem:2pointestimate} and~\ref{lem:expreform} we obtain
	\begin{align*}
		\E{Q_xQ_y\1(E>S)} \leq \left(1 -\frac{2C_d}{(1+p_d)r^{d-2}}(1+o(1)) \right)^{S} \leq  n^{-2\alpha d/(1+p_d) +o(1)}.
	\end{align*}

	For the second term on the right hand side of~\eqref{eq:upperboundonqx} we have
	\begin{align}\label{eq:essmall}
	\begin{split}
		\pr{E\leq S}  &\leq \pr{E\leq S, \sigma_y\geq \alpha t_*(1-\delta)^2} + \pr{\sigma_y<\alpha t_*(1-\delta)^2}\\
&		\leq  \pr{N_x(r,R,\alpha t_*(1-\delta)^2)\leq S} + \pr{N_y(r,R,\alpha t_*(1-\delta)^2)\geq A}\\
&=o( n^{-2\alpha d/(1+p_d) +o(1)}),
\end{split}
	\end{align}
	where the last equality follows from Lemma~\ref{lem:n1}.
	
	(b)  Define $L_x$ to be the number of excursions across $\til{B}_x=\B(x,n^{2\zeta/3})\setminus \B(x,n^{\zeta/3})$ that the walk has completed by time $\sigma_x$ and analogously define $L_y$. Set $\til{\delta}=n^{\psi+(2-d)\zeta/6}$ for $\psi$ sufficiently small and 
	\[
	M= \frac{1}{1+\til{\delta}}\cdot \frac{\alpha t_*(1-\til{\delta})^2}{T_{n^{\zeta/3}, n^{2\zeta/3}}}.
	\]
	Let $\Fexc$ be the sigma algebra generated by the entrance and exit points of the first $M$ excursions across $\B(x,n^{2\zeta/3})\setminus \B(x,n^{\zeta/3})$ and the first $M$ excursions across $\B(y,n^{2\zeta/3})\setminus \B(y,n^{\zeta/3})$. Then 
	\begin{align*}
		\E{Q_xQ_y} \leq \E{Q_xQ_y \1(L_x, L_y\geq M)} + 2 \pr{L_x<M}.
	\end{align*}
	The term $\pr{L_x<M}$ can be controlled in exactly the same way as in~\eqref{eq:essmall}. To bound the first term appearing on the right hand side above we define the events
	\[
	F_x = \{ x \text{ is not hit in $M$ excursions across } \til{B}_x\}
	\]
	and $F_y$ analogously. 
	After conditioning on entrance and exit points of the excursions, the events of hitting the centres are independent. Hence using Lemma~\ref{lem:hittingprob} we obtain
	\begin{align*}
		\E{Q_xQ_y \1(L_x, L_y\geq M)} &\leq \E{\prcond{F_x}{\Fexc}{}\prcond{F_y}{\Fexc}{}}\\& \leq \left(1-\frac{C_d}{n^{\zeta(d-2)/3}}(1+o(1)) \right)^{2M} = n^{-2\alpha d + o(1)},
	\end{align*}
	where for the last equality we used that 
		\[
	M= \frac{\alpha}{C_d}\cdot \log (n^d) n^{\zeta(d-2)/3}\cdot (1+o(1)),
	\]
	which follows from \cite[Lemma~4.1]{JasonPerla} proved in the Appendix of~\cite{JasonPerla}.
	\end{proof}

We now have all the required ingredients to prove Theorem~\ref{thm:mainresult}.

\begin{proof}[\bf Proof of Theorem~\ref{thm:mainresult}]
	
	By Lemma~\ref{lem:comparison} it suffices to show that 
	\[
	\norm{\LL(\overline{\U}) - \nu_{\alpha,n}}_{\rm{TV}} = o(1).
	\]
	In order to do so, we are going to use the Chen--Stein method as in Theorem~\ref{thm:chenstein}. For every~$x$ we define its neighbourhood $\B_x=\B(x,100R)$, set  $p_x=p_{\alpha,n} = \E{Q_0}$ and $p_{xy}=\E{Q_xQ_y}$ for~$x\neq y$. We now need to bound the terms $b_1, b_2$ and $b_3$.
	
Lemma~\ref{lem:q0precise} now gives as $n\to\infty$
	\begin{align*}
b_1= \sum_x \sum_{y\in \B_x\setminus \{x\}} p_x p_y \asymp n^d n^{\gamma d} n^{-2\alpha d} = n^{-\epsilon d} = o(1).
\end{align*}
Regarding the quantity $b_2$ we use Lemma~\ref{lem:b2quantity} to obtain 	
\begin{align*}
b_2 &= \sum_x \sum_{y: \norm{y-x}\leq n^\zeta} \E{Q_xQ_y} + \sum_x \sum_{y: n^\zeta\leq \norm{y-x}\leq n^\gamma} \E{Q_xQ_y} \\
&\asymp  n^d \cdot n^{\zeta d} \cdot n^{-2\alpha d/(1+p_d)+o(1)} + n^d \cdot n^{\gamma d} \cdot n^{-2\alpha d+o(1)}.
\end{align*}
Choosing $\zeta<2\alpha/(1+p_d) -1$ (recall that $\alpha >(1+p_d)/2)$) and since $\gamma = 2\alpha -1-\epsilon$, 
we get~$b_2 = o(1)$.
	We finally turn our attention to the quantity $b_3$. By transitivity, we have 
\[
b_3 = n^d\E{\left|\econd{Q_0 - p_0}{Q_y, y \notin \B_0} \right|}.
\]
We let $\F_{\rm{exc}}$ be the sigma algebra generated by the exit points of the first $A$ excursions across the annulus $\B(0,n^\gamma)\setminus \B(0,n^{\gamma(1-\epsilon)})$. We write $\F_{\rm{out}}$ for the sigma algebra generated by $\{Q_y, y\notin \B_0\}$. Then by the tower property we have 
\begin{align*}
	\econd{Q_0}{\Fout} = \econd{\econd{Q_0}{\sigma(\Fout,\Fexc)}}{\Fout} = \econd{\econd{Q_0}{\Fexc}}{\Fout},
\end{align*}
where for the last equality we used that $Q_0$ depends on $\Fout$ only through $\Fexc$ which follows from the fact that the annuli are disjoint by the choice of the radii. 
Using the same notation as in Lemma~\ref{lem:coupling}, we let
$Y_i$ be the exit point of the $i$-th excursion. Then 
\[\textstyle
\econd{Q_0}{\Fexc} = \prod_{i=1}^{A} f(Y_{i-1},Y_i),
\]
where $f$ was defined in~\eqref{eq:defoff}.
We then get  
\[\textstyle
\log \econd{Q_0}{\Fexc} = \sum_{i=1}^{A} \log f(Y_{i-1},Y_i).
\]
Let $\eta=n^{-\gamma(1-\epsilon)(d-2)/2}\log n$  and set $F$ to be the event 
\[
F=\left\{ \econd{Q_0}{\Fexc} \in \left( \exp(-m(1+\eta)A), \exp(-m(1-\eta)A) \right)\right\},
\]
where we recall from Lemma~\ref{lem:expreform} that 
\[
m=C_d\cdot n^{-\gamma(1-\epsilon)(d-2)}(1+O(n^{-\gamma \epsilon})).
\]
Then we obtain
\begin{align*}
	&\E{\left|\econd{\econd{Q_0}{\Fexc}}{\Fout} - \E{\econd{Q_0}{\Fexc}} \right|} \\
	&\leq  \E{\left|\econd{\econd{Q_0}{\Fexc}\1(F)}{\Fout} - \E{\econd{Q_0}{\Fexc}\1(F)}  \right|} +2\pr{F^c}\\
	&\leq   \exp(-m(1-\eta)A) - \exp(-m(1+\eta) A) + 2\pr{F^c}\\& 
	\lesssim e^{-m A}(e^{m \eta A} - e^{-m \eta A}) + R^{d-1} \exp\left(-c\eta^2 A \right),
\end{align*}
where the last inequality follows from Lemma~\ref{lem:q0precise}. Using~\eqref{eq:defoftstar} next gives $e^{-m A} \asymp n^{-\alpha d}$. Also using that $A\asymp n^{\gamma(1-\epsilon)(d-2)}\log n$ we obtain
\[
e^{m \eta A} - e^{-m \eta A}\asymp m \eta A \asymp n^{-\gamma(1-\epsilon) (d-2)/2} \cdot  (\log n)^2.
\]
So overall we obtain
\[
b_3\lesssim n^d n^{-\alpha d}n^{-\gamma (1-\epsilon)(d-2)/2}  (\log n)^2
\]
Substituting the value of $\gamma =2\alpha -1 - \epsilon$ we see that taking $\alpha >\frac34(d-\frac23)/(d-1)$ and $\epsilon>0$ sufficiently small gives~$b_3=o(1)$. This concludes the proof.
\end{proof}

\section*{Acknowledgements}

We thank Jason Miller and Ofer Zeitouni for helpful discussions. We are also grateful to the referee for their careful reading and comments. This work was supported by the Engineering and Physical Sciences Research Council:
	SOT by Doctoral Training Grant 1885554
and
	PS by EP/R022615/1.

\makeatletter
\def\@rst #1 #2other{#1}
\renewcommand\MR[1]{\relax\ifhmode\unskip\spacefactor3000 \space\fi
  \MRhref{\expandafter\@rst #1 other}{#1}}
\newcommand{\MRhref}[2]{\href{http://www.ams.org/mathscinet-getitem?mr=#1}{MR#2}}
\makeatother

\phantomsection
\bibliographystyle{hmralpha}
\bibliography{biblio}

\end{document}